\documentclass[a4paper,11pt]{article}
\usepackage[utf8]{inputenc}
\usepackage[T1]{fontenc}
\usepackage{lmodern}
\usepackage{microtype}
\usepackage[a4paper,textwidth=16cm,textheight=22cm,centering]{geometry}

\usepackage{mathtools}
\usepackage{amssymb,amsfonts}
\usepackage{amsthm}
\usepackage{setspace}
\usepackage{indentfirst}

\usepackage{cite}
\usepackage{xcolor}
\usepackage[ ]{hyperref}
\hypersetup{
	colorlinks=true,   
	linktoc=all,        
	linkcolor=red,     
	citecolor=blue,
	urlcolor=black,} 

\theoremstyle{plain}
\newtheorem{theorem}{Theorem}[section]
\newtheorem{lemma}[theorem]{Lemma}

\newtheorem{fact}[theorem]{Fact}

\newtheorem{observation}[theorem]{Observation}

\theoremstyle{definition}
\newtheorem{definition}[theorem]{Definition}
\newtheorem{problem}[theorem]{Problem}

\theoremstyle{remark}

\begin{document}
	\date{}
	\begin{spacing}{1.03}
\title{New Tower-Type Lower Bounds for Hypergraph Ramsey Numbers}
\author{Hanzhi Bai$^1$, \;\;\;
Longma Du$^1$, \;\;\; Xinyu Hu$^2$, \;\;\; Ruilong Liu$^1$, \;\;\; Guanghui Wang$^1$}

\footnotetext[1]{{School of Mathematics, Shandong University, Jinan 250100, P.~R.~China.
Emails: {\tt bhz@mail.sdu.edu.cn} (H. Bai),
{\tt 202520303@mail.sdu.edu.cn} (L. Du),
{\tt liuruilong@mail.sdu.edu.cn} (R. Liu),
{\tt ghwang@sdu.edu.cn} (G. Wang).
Supported by the National Natural Science Foundation of China (12231018) and the State Key Laboratory of Cryptography and Digital Economy Security.}}
\footnotetext[2]{Data Science Institute, Shandong University, Jinan 250100, P.~R.~China.
Email: {\tt huxinyu@sdu.edu.cn} (X. Hu).
Supported by National Postdoctoral Fellowship Program (C-tier) (GZC20252005).}

\maketitle	

\begin{abstract}

The Ramsey number $r_k(s,m)$ is the smallest $N$ such that any red/blue coloring of the $k$-subsets of $[N]$ contains a red $s$-set or a blue $m$-set. For fixed $k$ and $s$, and for sufficiently large $m$, the tower growth rate is determined by the stepping-up lemma, but for $s=m=k+1$ the available stepping-up lemmas do not apply. Fox asked for estimates of $r_k(k+1,k+1)$. Pudl\'ak, R\"odl, and Wesley gave the first tower-type bound:
$r_k(k+1,k+1)\ge s_3(\lfloor k/4\rfloor)\ge 4\operatorname{twr}_{\lfloor k/4\rfloor-4}(2)$, where $s_3(k)$ is the $3$-color shift number and $\operatorname{twr}_1(2)=2$, $\operatorname{twr}_{i+1}(2)=2^{\operatorname{twr}_i(2)}$.

In this paper, for $k\ge 6$, we improve the lower bound to $r_k(k+1,k+1)> s_3\bigl(\lfloor k/2\rfloor-2\bigr)$ by overcoming an obstruction in their construction. In addition, we give an exact characterization of $s_3(k)$ and, for $k\ge 5$, obtain a new explicit lower bound $s_3(k)\ge(\operatorname{twr}_{k-2}(2))^2$, which improves the result of Pudl\'ak and R\"odl. Consequently, for $k\ge 14$, $r_k(k+1,k+1)>(\operatorname{twr}_{\lfloor k/2\rfloor-4}(2))^2$.
\end{abstract}
	
\section{Introduction}
For a positive integer $n$, write $[n]:=\{1,2,\ldots,n\}$. A $k$-uniform hypergraph $H$ ($k$-graph for short) is a pair consisting of a set of vertices $V(H)$ and a collection of $k$-element subsets of $V(H)$. Let $K_n^{(k)}$ be the complete $k$-graph on $n$ vertices. The \textit{Ramsey number $r_k(s,m)$} \cite{R} is the smallest integer $N$ such that every red/blue coloring of the $k$-subsets of $[N]$ contains either an $s$-element subset all of whose $k$-subsets are red or an $m$-element subset all of whose $k$-subsets are blue.

Ramsey numbers have been extensively studied since 1935 \cite{E-S-1}, with many classical results \cite{A-K-S-1,A-K-S-2,B-K-2,C-G-E,E-H-Con,E-H-2,E-R-2,K-1,L-R-Z,M-S-3,M-S-4,M-Shapira,S-1,SP-1}. In recent years, there has been much new work on Ramsey numbers and related extremal problems, particularly in graphs \cite{B-1,C-G-M-S,C-J-M-S,C-F-S-3,F-H-L-L,G-N-N-W,H-H-K-P,H-L-L-W-1,H-M-S,M-S-X,M-V-2,P-G-M,L-N-1}, including several breakthroughs \cite{B-1,C-G-M-S,C-F-S-3,G-N-N-W,M-S-X,M-V-2}. The present paper, however, concerns only the case $k\ge 3$, and $k=2$ is therefore not covered.

For $s>k\ge 3$, the classical stepping-down argument \cite{E-R-2} asserts that there is a constant $c=c(k)>0$ such that $r_k(s,s)\le \operatorname{twr}_k(c\cdot s)$, where the tower function $\operatorname{twr}_k(x)$ is defined by $\operatorname{twr}_{1}(x)=x$ and $\operatorname{twr}_{i+1}(x)=2^{\operatorname{twr}_{i}(x)}$. Recently, Dob\'{a}k and Mulrenin \cite{D-M-1} improved the upper bound to $r_k(s,s)\le \operatorname{twr}_{k-1}\left(c\cdot s\binom{2(s-k+1)}{s-k+1}\right)$, where $c\ge 2$ is an absolute constant.

On the other hand, it is easy to obtain $r_3(s,s)\ge 2^{c\cdot s^2}$ by the basic probabilistic method, where $s\ge 3$ and $c>0$ is an absolute constant. More generally, applying the Erd\H{o}s--Hajnal stepping-up lemma \cite{E-H-R-1,G-R-S-1}, one obtains $r_k(s,s)\ge\operatorname{twr}_{k-1}({c'\cdot s^2})$, where $c'>0$ is an absolute constant. It is worth noting that the exponential gap between lower and upper bounds of $r_k(s,s)$ is a well-known difficult problem \cite{E-H-R-1} in hypergraph Ramsey theory. However, this lower bound for $r_k(s,s)$ was known to hold only when $s$ is sufficiently large with respect to $k$.
This discrepancy arises because the original stepping-up lemma from \cite{E-H-R-1} imposes the constraint $s\ge s_0(k)$, where $s_0(k)$ grows exponentially in $k$. Conlon, Fox, and Sudakov \cite{C-F-S-1} later improved the stepping-up lemma. Although the improved version removes the exponential restriction, it still requires $s\ge \frac{5}{2}k+4$. When $s$ is closer to $k+1$, the available stepping-up lemmas do not apply.

For off-diagonal Ramsey numbers ($s\neq m$), Conlon, Fox, and Sudakov \cite{C-F-S-3} obtained that for every $m\ge s\ge4$, $2^{c\cdot s m\log (\frac{m}{s}+1)}\le r_3(s,m)\le 2^{(c'\frac{m}{s})^{s-2}\log (\frac{m}{s})}$, where $c,c'>0$ are absolute constants.
For $s,m>k\ge 4$, it is known that there is a constant $c=c(k,s)>0$ such that $r_k(s,m)<\operatorname{twr}_{k-1}(m^c)$ \cite{E-R-2} (see also \cite[Theorem 10.1.4]{W-lecture}). Recently, Dob\'{a}k and Mulrenin \cite{D-M-1} also improved the upper bound to $r_k(s,m)\le \operatorname{twr}_{k-1}\left(c\cdot (s+m)\binom{2(s+m-2k+2)}{s-k+1}\right)$, where $c\ge 1$ is an absolute constant.

For the lower bound, Mubayi and Suk \cite{M-S-4}, as well as Conlon, Fox, and Sudakov, independently proved that there is an absolute constant $c>0$ such that $r_k(k+3,m)\ge \operatorname{twr}_{k-1}(c\cdot m)$ for $k\ge 4$ and $m>3k$. For $s\in\{k+1,k+2\}$, Mubayi and Suk \cite{M-S-3} established the lower bounds $r_k(k+2,m)\ge \operatorname{twr}_{k-1}({cm^{1/5}})$ and $r_k(k+1,m)\ge \operatorname{twr}_{k-2}(m^{c'\log m})$, where $m\ge k!2^{k}$ and $c,c'>0$ may depend on $k$. Recently, Du, Hu, Liu, and Wang \cite{D-H-L-W-1,D-H-L-W-2} obtained that $r_k(k+2,m)\ge \operatorname{twr}_{k-1}({cm^{1/2}})$ and $r_k(k+1,m)\ge \operatorname{twr}_{k-1}({c'm^{1/7}})$, where $m\ge k!2^{k}$ and $c,c'>0$ may depend on $k$. Together, these results determine the tower height for all classical off-diagonal hypergraph Ramsey numbers. Shortly after, Fan, Li, Lin, and Ning \cite{F-L-L-N}, as well as Du, Hu, Liu, and Wang, independently improved this to $r_k(k+1,m)\ge \operatorname{twr}_{k-1}({cm^{1/5}})$, where $m\ge k!2^{k}$ and $c>0$ depends on $k$. The exponential lower restriction on $m$ also stems from a limitation of the stepping-up lemma \cite[Lemma 2.1]{M-S-3}.

Therefore, determining $r_k(k+1,m)$ when $m$ is small is a natural problem. In fact, Fox \cite[Problem 1.32]{Graph-Ramsey-theory} posed the following problem.
\begin{problem}[Fox \cite{Graph-Ramsey-theory}]\label{center-problem}
     Estimate $r_k(k+1,k+1)$.
\end{problem}

As mentioned above, the upper bound of \( r_k(k+1,k+1) \) is \( \operatorname{twr}_{k-1}({ck}) \) for some absolute constant \( c
> 0 \). For the lower bound, a simple application of the probabilistic method yields $r_k(k+1,k+1)\ge 2^{(k+1)^{k-1}/{k^k}}$ for all $k\ge 2$ (see \cite[Proposition 10.2.1]{W-lecture}). Recently, Pudl\'{a}k, R\"{o}dl, and Wesley \cite{P-R-W-1} proved that \( r_k(k+1,k+1) \ge 4 \operatorname{twr}_{\lfloor k/4 \rfloor - 4}(2) \) for all \( k \ge 16 \)\footnote{Note that \cite{P-R-W-1} wrote \( r_k(k+1,k+1) \ge 4 \operatorname{twr}_{\lfloor k/4 \rfloor - 3}(2) \); the subscript \(\lfloor k/4\rfloor-3\) should be replaced by \(\lfloor k/4\rfloor-4\), giving the bound \( r_k(k+1,k+1) \ge 4\operatorname{twr}_{\lfloor k/4 \rfloor - 4}(2) \).} via $3$-colorings of shift graphs, which provides the first tower-type lower bound.
\medskip

Apart from the stepping-up lemma, another method to derive a tower-type lower bound for the Ramsey number of $k$-graphs is to use lower bounds on Ramsey numbers for monotone paths. For a sequence of vertices $\{v_i\}_{i=1}^{m+k-1}\subset [N]$ with $v_1<v_2<\cdots <v_{m+k-1}$, we say that the edges
$\{v_i, \ldots, v_{i+k-1}\}_{i=1}^m$ form a \textit{monotone path}, and we refer to the number of edges as its length (so the path above has length $m$). Let $P_k(s,m)$ be the smallest integer $N$ such that every red/blue coloring of the $k$-subsets of $[N]$ contains either a red monotone path of length $s$ or a blue monotone path of length $m$. It follows from the definitions of $r_k(s,m)$ and $P_k(s,m)$ that $P_k(s,m)\le r_k(k+s-1,k+m-1)$. In \cite{P-R-W-1}, Pudl\'{a}k, R\"{o}dl, and Wesley obtained that $\operatorname{twr}_{k-1}\left(\frac{m}{2\sqrt{2}}\right)\le P_k(m,m)\le \operatorname{twr}_{k-1}(2m)$ for $k\ge 3$, where the lower bound holds for $m\ge 6$, while the upper bound holds already for $m\ge 2$. Therefore, $r_k(k+5,k+5)\ge P_k(6,6)\ge \operatorname{twr}_{k-1}\left(\frac{3}{\sqrt{2}}\right)$. For the off-diagonal path Ramsey numbers ($s\neq m$), Milans, Stolee, and West \cite{M-S-W} used combinatorial tools, namely chains and antichains, to give an implicit lower bound for $P_k(s,m)$, which in turn yields a lower bound for $r_k(k+s-1,k+m-1)$. Even though the lower bounds obtained in this way are smaller than those obtained by the stepping-up lemma, the method works in some ranges where the stepping-up lemma does not. For instance, the stepping-up lemma fails to apply to $r_k(k+1,k+1)$, whereas $P_k(2,2)\le 2k+1$ and \( r_k(k+1,k+1) \ge 4 \operatorname{twr}_{\lfloor k/4 \rfloor - 4}(2) \) for all \( k \ge 16 \).
\medskip

The directed shift graph $\operatorname{Sh}(N,k)$ has vertex set $\binom{[N]}{k}$. For every $x_1<\cdots<x_{k+1}$, it contains the arc
\(
 \{x_1,\ldots,x_k\}\longrightarrow\{x_2,\ldots,x_{k+1}\}.
\)
Whenever we discuss chromatic number, we pass to the underlying undirected graph. For a graph $G$, let $\chi(G)$ denote its chromatic number. For each fixed $k$, Ramsey's theorem implies that $\chi(\operatorname{Sh}(N,k))\to\infty$ as $N\to\infty$. Such colorings of large shift graphs with few colors have been used to estimate Ramsey numbers (see \cite{D-L-R,P-R-W-1,P-Rodl-1}), and we will also employ them here. For this purpose, we define the \textit{$3$-color shift number} $s_3(k)$ by
\[
s_3(k) := \max\{ N:\chi(\operatorname{Sh}(N,k)) \le 3\}.
\]

Without loss of generality, assume first that $4\mid k$. In \cite{P-R-W-1}, Pudl\'ak, R\"odl, and Wesley used the properties of $s_3(k/4)$ to construct a $2$-coloring of $\binom{[s_3(k/4)]}{k}$ with no monochromatic $K_{k+1}^{(k)}$ as follows. Let $\phi$ be a $3$-coloring of $\operatorname{Sh}(s_3(k/4),k/4)$, whose existence is guaranteed by $s_3(k/4)$. Divide each $X\in\binom{[s_3(k/4)]}{k}$ into consecutive segments of equal size $k/4$, say $X=(X_1,X_2,X_3,X_4)$. Define $\lambda(X)=(\phi(X_1),\phi(X_2),\phi(X_3),\phi(X_4))\in[3]^4$. They then found a $2$-coloring $\psi:[3]^4\rightarrow[2]$ such that for every $(k+1)$-element subset $Y\subset[s_3(k/4)]$, there exist \textit{five special $k$-subsets} of $Y$ on which $\psi\circ\lambda:\binom{[s_3(k/4)]}{k}\rightarrow[2]$ is not monochromatic, and thus $\psi\circ\lambda$ contains no monochromatic $K_{k+1}^{(k)}$. They called the image of these five special $k$-subsets under $\lambda$ a \textit{bridge} $\beta_4$. They generalized $\beta_4$ to $\beta_n$ and showed that if $\beta_n$ is $2$-colorable, then $r_k(k+1,k+1)\ge s_3(\lfloor k/n \rfloor)$. In particular, they proved that $\beta_4$ is $2$-colorable, while $\beta_3$ is not. Thus, in their framework of ``monochromatic blocks and bridge'', $\beta_4$ is the minimal feasible bridge, and so $r_k(k+1,k+1)\ge s_3(\lfloor k/4 \rfloor)$. On the other hand, Pudl\'{a}k and R\"{o}dl \cite{P-Rodl-1} proved that $s_3(k)\ge 4 \operatorname{twr}_{k-4}(2)$. Thus, $r_k(k+1,k+1)\ge s_3(\lfloor k/4 \rfloor)\ge 4 \operatorname{twr}_{\lfloor k/4 \rfloor - 4}(2)$ \cite{P-R-W-1}.
\medskip 

In this paper, we first replace $\beta_4$ in \cite{P-R-W-1} by recording the directed color triples of three adjacent $k/2-2$-subsets within a block of length $k/2$; we first treat the case in which $k$ is even without loss of generality. This overcomes the obstruction that prevents the three-block bridge from being $2$-colorable in their framework and yields the following new lower bound for $r_k(k+1,k+1)$.
\begin{theorem}\label{cen-theorem-1}
For every $k\ge 6$, we have $r_k(k+1,k+1)> s_3\left(\lfloor k/2\rfloor-2\right)$.
\end{theorem}

In order to state our second result, we need terminology from posets. We view a partially ordered set, or \textit{poset}, as a set equipped with a partial order. Let \( (Q,\le) \) be a poset. An \textit{antichain} in $Q$ is a set of pairwise incomparable elements. The \emph{width} of a finite poset $Q$, denoted $w(Q)$, is defined as the cardinality of a \emph{maximum antichain} in $Q$, i.e., the largest size of any antichain contained in $Q$. A \textit{down-set} in \( Q \) is a set \( S \subseteq Q \) such that if \( y \in S \) and \( x \leq y \), then \( x \in S \). For $S\subseteq Q$, define its \emph{downward closure} as
\[
\downarrow S := \{ y \in Q \mid \exists~ x \in S \text{ with } y \le x \}.
\]
If $S = \emptyset$, we set $\downarrow \emptyset = \emptyset$. Let \( |Q| \) denote the number of elements in \( Q \). We write \( J(Q) \) for the containment poset of all down-sets of \( Q \). Let \(J^i(Q)\) denote the poset obtained from \(Q\) by repeatedly taking the down-set poset a total of \(i\) times. That is, \( J^0(Q) = Q \) and \( J^i(Q) = J(J^{i-1}(Q)) \) for \( i \geq 1 \). Let $A_3$ denote a three-element antichain and define $Q_0=A_3$, $Q_{i+1}=J(Q_i)$. Then
\[
 |Q_0| = 3,\quad |Q_1| = 8,\quad |Q_2| = 20,\quad |Q_3| = 84,\quad |Q_4| = 8573.
\]
For a recursive verification, define $e(P)=|J(P)|$. For any $x\in P$, partition $J(P)$ according to whether a down-set contains $x$. A down-set not containing $x$ contains no element of $\uparrow x$, and hence is a down-set of $P\setminus\uparrow x$. A down-set containing $x$ contains all of $\downarrow x$, and the map $I\mapsto I\setminus\downarrow x$ is a bijection from this class to $J(P\setminus\downarrow x)$. Therefore
\[
 e(P)=e(P\setminus \uparrow x)+e(P\setminus \downarrow x),
\]
where $\uparrow x=\{y\in P:x\le y\}$ and $\downarrow x=\{y\in P:y\le x\}$. This recurrence gives $|Q_3|=84$. In Kriegel's notation, $P_0^3$ is a three-element antichain and $P_{n+1}^3=\operatorname{Ideals}(P_n^3)$, so $Q_n\cong P_n^3$; the value $|Q_4|=8573$ is recorded in~\cite[Table~4.1]{Kriegel-1}.

In this paper, we also show that $s_3(k)=|J^{k-1}(A_3)|$. Moreover, for $k\ge 5$, we prove the explicit lower bound $s_3(k)=|J^{k-1}(A_3)|\ge (\operatorname{twr}_{k-2}(2))^2$, which improves the result $s_3(k)\ge 4\operatorname{twr}_{k-4}(2)$ of Pudl\'{a}k and R\"{o}dl~\cite{P-Rodl-1}.

\begin{theorem}\label{cen-theorem-2}
   Let $A_3$ be an antichain with $|A_3|=3$. Then, for every $k\ge 1$, we have $s_3(k)=|J^{k-1}(A_3)|$. Moreover, for every $k\ge 5$,
\[
{(\operatorname{twr}_{k-2}(2))^2 \le s_3(k)=|J^{k-1}(A_3)|\le \frac{\operatorname{twr}_{k-1}(2)}{2}.}
\]
\end{theorem}

\section{Exact values and bounds for $s_3(k)$}\label{section-s_k}
Let $G$ and $H$ be directed graphs. A \textbf{homomorphism} from $G$ to $H$ is a mapping $f: V(G) \to V(H)$ that preserves the direction of edges: for every directed edge $(u,v) \in E(G)$, the ordered pair $(f(u), f(v))$ belongs to $E(H)$. An \textbf{isomorphism} is a bijective homomorphism whose inverse is also a homomorphism; equivalently, a bijection $f: V(G) \to V(H)$ such that $(u,v) \in E(G)$ if and only if $(f(u), f(v)) \in E(H)$. If such a bijection exists, we say that $G$ is \textbf{isomorphic} to $H$ and write $G \cong H$.

Let $T_N$ denote the transitive tournament on $N$ vertices with vertex set $[N]$, equipped with a directed edge $\overrightarrow{ij}$ if and only if $i<j$. Thus, $T_N=\operatorname{Sh}(N,1)$. For a digraph $G$, let $\partial G$ be the digraph whose vertices are the directed edges of $G$. In $\partial G$, there is a directed edge from $(u, v)$ to $(v, w)$ if and only if $\overrightarrow{uvw}$ is a directed $2$-path in $G$.
For every $k\ge1$, there is a natural isomorphism $\operatorname{Sh}(N,k)\cong\partial^{k-1}T_N$. Indeed, by induction on $k$, the vertices of $\partial^{k-1}T_N$ are naturally identified with directed paths $\overrightarrow{x_1\cdots x_k}$ in $T_N$, equivalently with increasing $k$-tuples $x_1<\cdots<x_k$. Under this identification, an arc of $\partial^{k-1}T_N$ is precisely $(x_1,\ldots,x_k)\to(x_2,\ldots,x_{k+1})$.
\medskip

For a finite poset $P$, we define the digraph $\Gamma(P)$ by
\[
V(\Gamma(P)) = P,\qquad \overrightarrow{xy} \iff x \ngeq_P y.
\]
In particular, let $A_r$ denote an antichain of size $r$. Then $\Gamma(A_r) = \overrightarrow{K}_r$, where $\overrightarrow{K}_r$ is the loopless complete digraph, i.e., $\overrightarrow{xy}$ and $\overrightarrow{yx}$ for all $x\neq y\in V(\overrightarrow{K}_r)$.
We have the following fact.
\begin{fact}\label{fact-sh}
A proper $r$-coloring of the underlying undirected graph of $\operatorname{Sh}(N,k)$ is equivalent to a digraph homomorphism $\operatorname{Sh}(N,k) \to \overrightarrow{K}_r = \Gamma(A_r)$.
\end{fact}
Indeed, since $\Gamma(A_r)$ has no loops and has both arcs between every two distinct vertices, a map to $\Gamma(A_r)$ preserves every oriented edge exactly when adjacent vertices receive distinct colors.

In this section, we provide an exact characterization and new bounds for $s_3(k)$. Define the \textit{$r$-color shift number} by $s_r(k)= \max\{ N:\chi(\operatorname{Sh}(N,k)) \le r\}$. In fact, we obtain the following more general theorem, which gives an exact characterization of the $r$-color shift number. Specifically, the first part of Theorem \ref{cen-theorem-2} is the case $r = 3$.

\begin{theorem}\label{thm:Optimal $r$-coloring threshold}
    {Let $r \ge 2$ and $k\ge1$, and let $A_r$ be an antichain of size $r$. Then $\chi(\mathrm{Sh}(N,k)) \le r$ if and only if $ N \le |J^{k-1}(A_r)|.$ Consequently,  $s_r(k) = |J^{k-1}(A_r)|$.}
\end{theorem}

We need two lemmas.

\begin{lemma}\label{lem:Right adjoint arc graph}
    {For any finite digraph $G$ and any finite poset $P$, $\partial G \to \Gamma(P)$ is a digraph homomorphism if and only if $G \to \Gamma(J(P))$ is a digraph homomorphism.}
\end{lemma}

\noindent\textit{Proof of Lemma~\ref{lem:Right adjoint arc graph}.~}{
Suppose first that $F: \partial G \to \Gamma(P)$ is a homomorphism. That is, for every arc $\overrightarrow{uv}$ in $G$ we have $F(u,v)\in P$, and for every directed $2$-path $\overrightarrow{uvw}$ we have $F(u,v) \not\ge_P F(v,w)$. For each vertex $v$ of $G$, let
\[
 I_v := \downarrow\{ F(u,v): u\in V(G),\ \overrightarrow{uv}\in E(G)\},
\]
and set $I_v=\emptyset$ if $v$ has no in-arcs. Clearly $I_v\in J(P)$. We claim that $v\mapsto I_v$ is a homomorphism from $G$ to $\Gamma(J(P))$. Fix an arc $\overrightarrow{vw}\in E(G)$. Since $F(v,w)\in I_w$, it remains to show that $F(v,w)\notin I_v$. Otherwise, there exists an arc $\overrightarrow{uv}\in E(G)$ such that $F(v,w)\le_P F(u,v)$, contradicting the assumption that $F$ is a homomorphism $\partial G\to\Gamma(P)$. Hence $I_v\not\supseteq I_w$, and $v\mapsto I_v$ is a homomorphism $G\to\Gamma(J(P))$.

Conversely, suppose that $I:V(G)\to J(P)$ is a homomorphism from $G$ to $\Gamma(J(P))$, and write $I_v:=I(v)$. For every arc $\overrightarrow{uv}$, we have $I_u\not\supseteq I_v$; choose $p_{uv}\in I_v\setminus I_u$ and define $F(u,v):=p_{uv}$. We claim that $(u,v)\mapsto F(u,v)$ is a homomorphism from $\partial G$ to $\Gamma(P)$. Let $\overrightarrow{uvw}$ be a directed $2$-path. Then $p_{uv}\in I_v$ and $p_{vw}\notin I_v$. If $p_{uv}\ge_P p_{vw}$, then $p_{vw}\in I_v$ because $I_v$ is a down-set, a contradiction. Therefore $F(u,v)\not\ge_P F(v,w)$, so $F(u,v)\to F(v,w)$ is an arc in $\Gamma(P)$.}
\hfill$\Box$
\medskip

Recall that $T_N=\operatorname{Sh}(N,1)$ is a transitive tournament on $N$ vertices. We have the following lemma.

\begin{lemma}\label{lem:Threshold for transitive tournaments}
    For any finite poset  $P$, $T_N \to \Gamma(P) $ is a digraph homomorphism if and only if $ N \le |P|$.
\end{lemma}
\noindent\textit{Proof of Lemma~\ref{lem:Threshold for transitive tournaments}.~}
Assume that $f: T_N \to \Gamma(P)$ is a homomorphism, then for any $i<j$, we  have $f(i) \not\ge_P f(j)$ by noting that $\overrightarrow{ij}$ is an arc of $T_N$. In particular $f(i)\neq f(j)$, so $f$ is injective and consequently $N \le |P|$.

Conversely, assume $N \le |P|$. Consider a linear extension $p_1, p_2, \dots, p_{|P|}$ of $P$ such that $p_a \le_P p_b \Rightarrow a \le b$, and define $f(i)=p_i$ for $i\in [N]$.
For any $i<j$, we have  $p_i \not\ge_P p_j$, otherwise $p_j \le_P p_i$ would imply $j \le i$, a contradiction. Hence $\overrightarrow{f(i)f(j)}$ is an arc in $\Gamma(P)$ and $f$ is a homomorphism $T_N \to \Gamma(P)$. 
\hfill$\Box$
\medskip

We now prove Theorem~\ref{thm:Optimal $r$-coloring threshold} using Lemmas~\ref{lem:Right adjoint arc graph} and~\ref{lem:Threshold for transitive tournaments}.
\medskip

\noindent\textit{Proof of Theorem~\ref{thm:Optimal $r$-coloring threshold}.~}{It follows from Fact \ref{fact-sh} that $\mathrm{Sh}(N,k)$ is $r$-colorable if and only if there exists a digraph homomorphism $\mathrm{Sh}(N,k) \to \Gamma(A_r)$. Since $\mathrm{Sh}(N,k) \cong \partial^{k-1}\mathrm{Sh}(N,1)=\partial^{k-1}T_N$, this is equivalent to a homomorphism $\partial^{k-1}T_N\to\Gamma(A_r)$. Repeated applications of Lemma~\ref{lem:Right adjoint arc graph} give}
\[
{\partial^{k-1}T_N\to\Gamma(A_r)
\Longleftrightarrow \partial^{k-2}T_N\to\Gamma(J(A_r))
\Longleftrightarrow\cdots\Longleftrightarrow
T_N\to\Gamma(J^{k-1}(A_r)).}
\]
{By Lemma~\ref{lem:Threshold for transitive tournaments}, the last condition holds if and only if $N\le |J^{k-1}(A_r)|$.}
\hfill$\Box$
\medskip

It remains to prove the explicit estimate in Theorem~\ref{cen-theorem-2} for $k\ge 5$, namely
\[
{(\operatorname{twr}_{k-2}(2))^2 \le s_3(k)=|J^{k-1}(A_3)|\le \frac{\operatorname{twr}_{k-1}(2)}{2}.}
\]

We need the following three observations. For the sake of completeness of the paper, we provide their proofs.

\begin{observation} \label{obser-1}
For any finite poset \(P\), \(|J(P)| \le 2^{|P|}\).
\end{observation}

\noindent\textit{{Proof of Observation~\ref{obser-1}.}~}{ Note that \(J(P)\) is the set of down-sets of \(P\) and every down-set is a subset of \(P\), hence \(J(P) \subseteq 2^{P}\). Therefore}
$|J(P)| \le2^{|P|}$.
\hfill$\Box$

\begin{observation}\label{obser-2}
If a finite poset \(P\) contains an antichain of size \(m\), then \(|J(P)| \ge 2^{m}\).
\end{observation}
\noindent\textit{{Proof of Observation~\ref{obser-2}.}~}{
Let \(A = \{x_1,\dots,x_m\}\) be an antichain in \(P\). For any subset \(S \subseteq A\), define
\[
I_S = \downarrow S = \{ y \in P : \exists s \in S,\ y \le s \}.
\]
Clearly \(I_S\) is a down-set. If \(S_1 \neq S_2\), say \(x \in S_1\setminus S_2\), then because \(A\) is an antichain, \(x\) is not comparable with any element of \(S_2\); hence \(x \in I_{S_1}\) but \(x \notin I_{S_2}\). Thus distinct subsets give distinct down-sets. There are \(2^{m}\) subsets of \(A\), so \(|J(P)| \ge 2^{m}\).}
\hfill$\Box$

\begin{observation}\label{obser-3}
For any finite poset \(P\), \(J(P)\) contains an antichain of size at least \(\frac{|J(P)|}{|P|+1}\).
\end{observation}

\noindent\textit{{Proof of Observation~\ref{obser-3}.}~}{
For \(t = 0,1,\dots,|P|\), let
$\mathcal{A}_t = \{ I \in J(P) : |I| = t \}$.
If \(I_1,I_2 \in \mathcal{A}_t\) and \(I_1 \neq I_2\), then neither \(I_1 \subseteq I_2\) nor \(I_2 \subseteq I_1\). Hence \(\mathcal{A}_t\) is an antichain in \(J(P)\). The family \(\{\mathcal{A}_0,\dots,\mathcal{A}_{|P|}\}\) partitions \(J(P)\), so
\[
\sum_{t=0}^{|P|} |\mathcal{A}_t| = |J(P)|.
\]
The $|P|+1$ nonnegative integers $|\mathcal{A}_0|,\ldots,|\mathcal{A}_{|P|}|$ sum to $|J(P)|$. By the pigeonhole principle, there exists some \(t\) such that
\[
|\mathcal{A}_t| \ge \frac{|J(P)|}{|P|+1}.
\]
This \(\mathcal{A}_t\) is the desired antichain.}
\hfill$\Box$
\medskip

\noindent\textit{{Proof of the second part of Theorem~\ref{cen-theorem-2}.}~}
Recall that $Q_0=A_3$, $Q_{i+1}=J(Q_i)$, and $|Q_0| = 3$, $|Q_1| = 8$, $|Q_2| = 20$, $|Q_3| = 84$, $|Q_4| = 8573$. For each $i\ge 4$, it follows from Observations~\ref{obser-2} and~\ref{obser-3} that $|J(Q_i)|\ge 2^{w(Q_i)}=2^{w(J(Q_{i-1}))}\ge 2^{\frac{|J(Q_{i-1})|}{|Q_{i-1}|+1}}=2^{\frac{|Q_i|}{|Q_{i-1}|+1}}$. We will prove by induction that
\begin{align}\label{ditui}
    \frac{|Q_i|}{|Q_{i-1}|+1}\ge 2\operatorname{twr}_{i-1}(2)\qquad\text{and}\qquad |Q_i|<\frac{\operatorname{twr}_{i}(2)}{2}.
\end{align}

For the base case $i=4$, we have $\frac{|Q_4|}{|Q_{3}|+1}=\frac{8573}{85}>100>32=2\operatorname{twr}_{3}(2)$ and $|Q_4|=8573<32768=\frac{\operatorname{twr}_{4}(2)}{2}$. Suppose that \eqref{ditui} holds for $i$. We now consider $i+1$. Since
\[
 |Q_{i+1}|=|J(Q_i)|\ge 2^{\frac{|Q_i|}{|Q_{i-1}|+1}}\ge 2^{2\operatorname{twr}_{i-1}(2)}=(\operatorname{twr}_{i}(2))^2,
\]
and since both $|Q_i|+1$ and $\operatorname{twr}_{i}(2)/2$ are integers, the inequality $|Q_i|<\operatorname{twr}_{i}(2)/2$ implies $|Q_i|+1\le \operatorname{twr}_{i}(2)/2$. Thus
\[
 \frac{|Q_{i+1}|}{|Q_{i}|+1}\ge \frac{(\operatorname{twr}_{i}(2))^2}{\operatorname{twr}_{i}(2)/2}=2\operatorname{twr}_{i}(2).
\]
On the other hand, Observation~\ref{obser-1} yields
\[
 |Q_{i+1}|=\bigl|J(Q_i)\bigr|\le 2^{|Q_i|}<2^{\operatorname{twr}_{i}(2)/2}<2^{\operatorname{twr}_{i}(2)-1}=\frac{\operatorname{twr}_{i+1}(2)}{2}.
\]
Hence $(\operatorname{twr}_{i-1}(2))^2 \le |Q_i| < \frac{\operatorname{twr}_{i}(2)}{2}$ for all $i\ge 4$. Taking $i=k-1$ gives, for every $k\ge 5$, $(\operatorname{twr}_{k-2}(2))^2 \le |J^{k-1}(A_3)|=|Q_{k-1}|<\frac{\operatorname{twr}_{k-1}(2)}{2}$.
\hfill$\Box$

\section{A new lower bound for \(r_k(k+1,k+1)\)}\label{section-ramsey}
In this section, we use the shift graph to give a stronger lower bound for $r_k(k+1,k+1)$. Specifically, we prove Theorem~\ref{cen-theorem-1} by showing that \(r_k(k+1,k+1)> s_3(\lfloor k/2\rfloor-2)\).
Throughout the proof, we assume $k\ge 6$, so $\ell=\lfloor k/2\rfloor-2\ge 1$.

First suppose that $k$ is even, so that \(k=2(\ell+2)\). Recall that \(s_3(\ell)=\max\{N:\chi(\mathrm{Sh}(N,\ell))\le 3\}\) and \(N=s_3(\ell)\). Fix a proper $3$-coloring \(\phi:\binom{[N]}{\ell}\to\{0,1,2\}\) of the shift graph \(\mathrm{Sh}(N,\ell)\), whose existence is guaranteed by $s_3(\ell)$. Let \(B=\{z_1<\cdots<z_{\ell+2}\}\) be an \((\ell+2)\)-subset listed in increasing order, and define its \textit{state vector} by
\[
V_s(B)=\bigl(\phi(\{z_1,\ldots,z_\ell\}),\phi(\{z_2,\ldots,z_{\ell+1}\}),\phi(\{z_3,\ldots,z_{\ell+2}\})\bigr).
\]
Since any two consecutive \(\ell\)-subsets appearing here are adjacent in \(\mathrm{Sh}(N,\ell)\), such vectors lie in
\[
D:=\{(a_0,a_1,a_2)\in\{0,1,2\}^3:a_0\ne a_1,\ a_1\ne a_2\}.
\]
For $p=(a_0,a_1,a_2)$ and $q=(b_0,b_1,b_2)$ in $D$, write $p\leadsto q$ if and only if $(b_1,b_2)=(a_0,a_1)$. Equivalently, $q=(c,a_0,a_1)$ for some $c\in\{0,1,2\}$. This transition records what happens when a length-\((\ell+2)\) window is shifted inside a length-\((\ell+3)\) window. We now give the following definition to describe a special structure for such transitions.

\begin{definition}\label{def bond}
A \emph{bond} in \(D^2\) is a triple of the form
\[
(p_1,p_2),\quad (q_1,p_2),\quad (q_1,q_2),
\]
where $p_1,p_2,q_1,q_2\in D$ and \(p_1\leadsto q_1\), \(p_2\leadsto q_2\) are transitions.
\end{definition}

To simplify the notation for elements of $D$, we also need the following definition.

\begin{definition}\label{def matrix}
Define the index map \(\iota:D\to\{0,1,\ldots,11\}\) by
\[
\begin{array}{c|cccccccccccc}
x & 010 & 012 & 020 & 021 & 101 & 102 & 120 & 121 & 201 & 202 & 210 & 212\\
\hline
\iota(x) & 0 & 1 & 2 & 3 & 4 & 5 & 6 & 7 & 8 & 9 & 10 & 11
\end{array}
\]
where, for example, \(010\) denotes the vector \((0,1,0)\) in $D$.
\end{definition}

Define four auxiliary maps \(\alpha,\beta,\gamma,\delta:\{0,\ldots,11\}\to\mathbb F_2\), where $\mathbb F_2=\{0,1\}$ is the field of two elements, as follows: \(\alpha(i)=0\) for \(i\in\{0,1,2,3,10,11\}\) and \(\alpha(i)=1\) otherwise; \(\beta(i)=1\) for \(i\in\{3,6\}\) and \(\beta(i)=0\) otherwise; \(\gamma(j)=0\) for \(j\in\{0,1,2,3,6,7\}\) and \(\gamma(j)=1\) otherwise; \(\delta(j)=1\) for \(j\in\{1,10\}\) and \(\delta(j)=0\) otherwise.

Now, we define the two-coloring \(\Phi:D^2\to\{0,1\}\) by

\begin{align}\label{phi-equ}
    \Phi(x,y)=\alpha(\iota(x))+\gamma(\iota(y))+\beta(\iota(x))\delta(\iota(y)),
\end{align}
where all operations are in \(\mathbb F_2\).

The next lemma claims that there is no monochromatic bond, and we will use it to prove that there is no monochromatic $K_{k+1}^{(k)}$ under the final coloring $\chi$ which will be defined later.

\begin{lemma}\label{lemma no monochromatic bond}
{No bond is monochromatic under \(\Phi\). That is, suppose that $p_1,p_2,q_1,q_2\in D$ and \(p_1\leadsto q_1\), \(p_2\leadsto q_2\) are transitions. Then the following cannot occur:}
\[
{\Phi(p_1,p_2)=\Phi(q_1,p_2)=\Phi(q_1,q_2).}
\]

\end{lemma}

\begin{proof}
Using the index map from Definition~\ref{def matrix}, write \(\iota(p_1)=i\), \(\iota(q_1)=i'\), \(\iota(p_2)=j\), and \(\iota(q_2)=j'\), with \(i,i',j,j'\in\{0,\ldots,11\}\). Put \((\alpha(i),\beta(i))=(a,b)\), \((\alpha(i'),\beta(i'))=(a',b')\), \((\gamma(j),\delta(j))=(c,d)\), and \((\gamma(j'),\delta(j'))=(c',d')\). By the definition of $\leadsto$ and the index map in Definition~\ref{def matrix}, the possible row-feature transitions \((a,b)\to(a',b')\) are exactly
\[
\{(0,0)\to(1,0),\ (0,1)\to(1,0),\ (1,0)\to(0,0),\ (1,1)\to(0,0),\ (1,0)\to(1,1),\ (0,0)\to(0,1)\}.
\]
The same list also describes the possible column-feature transitions \((c,d)\to(c',d')\). For reference, Appendix~\ref{app:verification} records the twelve-state table that verifies both lists.

Assume for contradiction that the bond is monochromatic. From \(\Phi(p_1,p_2)=\Phi(q_1,p_2)\) and (\ref{phi-equ}), we obtain
$
a+c+bd = a'+c+b'd,
$
and hence 
\begin{align}\label{equal-1}
    (a+a') + (b+b')d = 0
\end{align}
(all operations in \(\mathbb F_2\)).
Similarly, the equality \(\Phi(q_1,p_2)=\Phi(q_1,q_2)\) gives
\begin{align}\label{equal-2}
(c+c')+b'(d+d')=0.
\end{align}
We now consider the possible row-feature transitions.

First suppose the row transition is \((0,0)\to(1,0)\) or \((1,0)\to(0,0)\). In either case \(a+a'=1\) and \(b+b'=0\). Substituting these values into (\ref{equal-1}) yields \(1=0\), which is impossible.

Next suppose the row transition is \((0,1)\to(1,0)\) or \((1,1)\to(0,0)\). Then \(a+a'=1\), \(b+b'=1\), and \(b'=0\). The equation (\ref{equal-1}) therefore forces \(d=1\). Hence the column transition starts from a feature whose second coordinate is \(1\). Among the six possible transition types listed above, this means that the column transition must be either \((0,1)\to(1,0)\) or \((1,1)\to(0,0)\), so \(c+c'=1\) and \(d+d'=1\). Substituting into (\ref{equal-2}) gives \(1+0\cdot 1=0\), again impossible.

Finally suppose the row transition is \((1,0)\to(1,1)\) or \((0,0)\to(0,1)\). Then \(a+a'=0\), \(b+b'=1\), and \(b'=1\). The equation (\ref{equal-1}) forces \(d=0\). Hence the column transition starts from a feature whose second coordinate is \(0\). If the column transition is \((0,0)\to(1,0)\) or \((1,0)\to(0,0)\), then \(c+c'=1\) and \(d+d'=0\), so the equation (\ref{equal-2}) gives \(1+1\cdot 0=0\), impossible. If the column transition is \((0,0)\to(0,1)\) or \((1,0)\to(1,1)\), then \(c+c'=0\) and \(d+d'=1\), so the equation (\ref{equal-2}) gives \(0+1\cdot 1=0\), also impossible.

All possible row-feature transitions lead to contradictions. Therefore the three points of a bond cannot have the same \(\Phi\)-color.
\end{proof}
\medskip

We now define the final coloring $\chi$ and prove that there is no monochromatic $K_{k+1}^{(k)}$ under $\chi$. Recall that $k=2(\ell+2)$. For $X=\{x_1<\cdots<x_k\}\in\binom{[N]}{k}$, split $X$ into the two consecutive blocks
\[
 X_1=\{x_1,\ldots,x_{\ell+2}\},
 \qquad
 X_2=\{x_{\ell+3},\ldots,x_k\},
\]
and set
\[
 \Lambda(X):=(V_s(X_1),V_s(X_2))\in D^2,
 \qquad
 \chi(X):=\Phi(\Lambda(X)).
\]
We claim that $\chi$ has no monochromatic $(k+1)$-set. Let \(Y=\{y_1<\cdots<y_{k+1}\}\). Consider the three special \(k\)-subsets \(Z_0=Y\setminus\{y_1\}\), \(Z_1=Y\setminus\{y_{\ell+3}\}\), and \(Z_2=Y\setminus\{y_{k+1}\}\), where $k+1=2\ell+5$. Put
\[
\begin{aligned}
p_1&=V_s(\{y_2,\ldots,y_{\ell+3}\}),&
q_1&=V_s(\{y_1,\ldots,y_{\ell+2}\}),\\
p_2&=V_s(\{y_{\ell+4},\ldots,y_{2\ell+5}\}),&
q_2&=V_s(\{y_{\ell+3},\ldots,y_{2\ell+4}\}).
\end{aligned}
\]
Then $p_1\leadsto q_1$ and $p_2\leadsto q_2$; this follows by comparing the three consecutive $\ell$-windows in the two adjacent blocks. Moreover,
\[
\Lambda(Z_0)=(p_1,p_2),\qquad \Lambda(Z_1)=(q_1,p_2),\qquad \Lambda(Z_2)=(q_1,q_2).
\]
Thus $\Lambda(Z_0),\Lambda(Z_1),\Lambda(Z_2)$ form a bond, so by Lemma~\ref{lemma no monochromatic bond} they cannot all have the same \(\Phi\)-color. Hence \(Z_0,Z_1,Z_2\) are not all assigned the same \(\chi\)-color, and \(Y\) is not monochromatic.

Consequently, when \(k=2(\ell+2)\), the coloring \(\chi:\binom{[s_3(\ell)]}{k}\to\{0,1\}\) contains no monochromatic \(K_{k+1}^{(k)}\), and hence
\[
r_k(k+1,k+1)>s_3(\ell).
\]
It remains to pass to arbitrary \(k\). Let \(\ell=\lfloor k/2\rfloor-2\) and \(k^\ast=2(\ell+2)=2\lfloor k/2\rfloor\le k\); then $k-k^\ast\in\{0,1\}$. By the even case, there is a coloring \(\chi_0:\binom{[s_3(\ell)]}{k^\ast}\to\{0,1\}\) with no monochromatic \((k^\ast+1)\)-set. For a \(k\)-set \(X=\{x_1<\cdots<x_k\}\), define its color to be \(\chi_0(\{x_1,\ldots,x_{k^\ast}\})\). Suppose that a \((k+1)\)-set \(Y=\{y_1<\cdots<y_{k+1}\}\) were monochromatic for this induced coloring, and set \(Y_0=\{y_1,\ldots,y_{k^\ast+1}\}\). For each $1\le t\le k^\ast+1$, the $k$-set $Y\setminus\{y_t\}$ has first $k^\ast$ vertices equal to $Y_0\setminus\{y_t\}$. Hence all $k^\ast$-subsets $Y_0\setminus\{y_t\}$, $1\le t\le k^\ast+1$, have the same $\chi_0$-color. This makes $Y_0$ a monochromatic \((k^\ast+1)\)-set under \(\chi_0\), contradicting the choice of \(\chi_0\). Consequently,
\[
r_k(k+1,k+1)>s_3\!\left( \lfloor k/2 \rfloor-2\right).
\]
This completes the proof of Theorem~\ref{cen-theorem-1}.

\section{Concluding remarks}

In this paper, we improve the lower bound for \(s_3(k)\). In fact, this improvement can be used not only to estimate $r_k(k+1,k+1)$, but also to estimate some other Ramsey numbers. In \cite{P-R-W-1},  Pudl\'{a}k, R\"{o}dl, and Wesley obtained that
$r_k(k+1,k+2)\ge s_3(k-4)\ge 4\operatorname{twr}_{k-8}(2)$,
$r_k(k+2,k+2)\ge s_3(k-1)\ge 4\operatorname{twr}_{k-5}(2)$ and 
$r_k(k+1,2k+1)\ge s_3(k)\ge 4\operatorname{twr}_{k-4}(2)$. 
By Theorem~\ref{cen-theorem-2}, these lower bounds can be improved in the same way, within the ranges where the corresponding $s_3$ estimate applies. In particular, 
$r_k(k+1,k+2)\ge s_3(k-4)\ge (\operatorname{twr}_{k-6}(2))^2$ for $k\ge 9$, 
$r_k(k+2,k+2)\ge s_3(k-1)\ge (\operatorname{twr}_{k-3}(2))^2$ for $k\ge 6$, and $r_k(k+1,2k+1)\ge s_3(k)\ge (\operatorname{twr}_{k-2}(2))^2$ for $k\ge 5$.

In addition, we believe that, by further optimizing the bond structure in Section \ref{section-ramsey}, the lower bound for \(r_k(k+1,k+1)\) may still be improved within the framework of the shift graph method. Therefore, we propose the following open problem.
\begin{problem}
{Does there exist an absolute constant $c\ge0$ such that, for every $k\ge c+1$, $r_k(k+1,k+1)\ge s_3(k-c)$?}
\end{problem}

\textbf{Declaration on the use of generative AI}

The authors used generative AI tools { (ChatGPT 5.5 Pro and 5.5 Thinking)} to assist in numerical computation, checking proofs and improving exposition.

\section*{Appendix}\label{app:verification}

\subsection*{Verification of the feature transitions}
For $i\in\{0,\ldots,11\}$, put
\[
R(i):=(\alpha(i),\beta(i))\qquad\text{and}\qquad C(i):=(\gamma(i),\delta(i)).
\]
We now verify, without omitting any case, the transition lists used in the proof of Lemma~\ref{lemma no monochromatic bond}.

\medskip
\noindent\textit{Step 1: enumeration of the states.}
Every $x=(a_0,a_1,a_2)\in D$ satisfies $a_0\ne a_1$ and $a_1\ne a_2$. There are three choices for $a_0$, two choices for $a_1$, and two choices for $a_2$. Hence $|D|=3\cdot2\cdot2=12$, and the twelve states in the table below exhaust $D$.

\medskip
\noindent\textit{Step 2: enumeration of the successors.}
Fix $x=(a_0,a_1,a_2)\in D$. By definition, $x\leadsto y$ if and only if $y=(c,a_0,a_1)$ for some $c\in\{0,1,2\}$. The condition $y\in D$ requires $c\ne a_0$, while $a_0\ne a_1$ already holds because $x\in D$. Thus there are exactly two choices for $c$. Consequently, every state has exactly two successors, and the transition digraph on $D$ has exactly $12\cdot2=24$ directed transitions.

\medskip
\noindent\textit{Step 3: evaluation of the features.}
The values of $R(i)$ and $C(i)$ follow directly from the definitions of $\alpha,\beta,\gamma$, and $\delta$. For each state, the successor column lists both possible successors together with their indices, and the last two columns record the induced row- and column-feature transitions.

\begin{center}
\begingroup\small
\setlength{\tabcolsep}{3.6pt}
\renewcommand{\arraystretch}{1.12}
\begin{tabular}{c c c c c c c}
\hline
$i$ & $x$ & $R(i)$ & $C(i)$ & successors $i'{:}x'$ & induced $R$-transitions & induced $C$-transitions \\
\hline
0  & 010 & $(0,0)$ & $(0,0)$ & $4{:}101,\ 8{:}201$   & $(0,0)\to(1,0),\ (0,0)\to(1,0)$ & $(0,0)\to(1,0),\ (0,0)\to(1,0)$ \\
1  & 012 & $(0,0)$ & $(0,1)$ & $4{:}101,\ 8{:}201$   & $(0,0)\to(1,0),\ (0,0)\to(1,0)$ & $(0,1)\to(1,0),\ (0,1)\to(1,0)$ \\
2  & 020 & $(0,0)$ & $(0,0)$ & $5{:}102,\ 9{:}202$   & $(0,0)\to(1,0),\ (0,0)\to(1,0)$ & $(0,0)\to(1,0),\ (0,0)\to(1,0)$ \\
3  & 021 & $(0,1)$ & $(0,0)$ & $5{:}102,\ 9{:}202$   & $(0,1)\to(1,0),\ (0,1)\to(1,0)$ & $(0,0)\to(1,0),\ (0,0)\to(1,0)$ \\
4  & 101 & $(1,0)$ & $(1,0)$ & $0{:}010,\ 10{:}210$  & $(1,0)\to(0,0),\ (1,0)\to(0,0)$ & $(1,0)\to(0,0),\ (1,0)\to(1,1)$ \\
5  & 102 & $(1,0)$ & $(1,0)$ & $0{:}010,\ 10{:}210$  & $(1,0)\to(0,0),\ (1,0)\to(0,0)$ & $(1,0)\to(0,0),\ (1,0)\to(1,1)$ \\
6  & 120 & $(1,1)$ & $(0,0)$ & $1{:}012,\ 11{:}212$  & $(1,1)\to(0,0),\ (1,1)\to(0,0)$ & $(0,0)\to(0,1),\ (0,0)\to(1,0)$ \\
7  & 121 & $(1,0)$ & $(0,0)$ & $1{:}012,\ 11{:}212$  & $(1,0)\to(0,0),\ (1,0)\to(0,0)$ & $(0,0)\to(0,1),\ (0,0)\to(1,0)$ \\
8  & 201 & $(1,0)$ & $(1,0)$ & $2{:}020,\ 6{:}120$   & $(1,0)\to(0,0),\ (1,0)\to(1,1)$ & $(1,0)\to(0,0),\ (1,0)\to(0,0)$ \\
9  & 202 & $(1,0)$ & $(1,0)$ & $2{:}020,\ 6{:}120$   & $(1,0)\to(0,0),\ (1,0)\to(1,1)$ & $(1,0)\to(0,0),\ (1,0)\to(0,0)$ \\
10 & 210 & $(0,0)$ & $(1,1)$ & $3{:}021,\ 7{:}121$   & $(0,0)\to(0,1),\ (0,0)\to(1,0)$ & $(1,1)\to(0,0),\ (1,1)\to(0,0)$ \\
11 & 212 & $(0,0)$ & $(1,0)$ & $3{:}021,\ 7{:}121$   & $(0,0)\to(0,1),\ (0,0)\to(1,0)$ & $(1,0)\to(0,0),\ (1,0)\to(0,0)$ \\
\hline
\end{tabular}
\endgroup
\end{center}

For example, if $i=8$, then $x=201=(2,0,1)$. Its successors have the form $(c,2,0)$ with $c\ne2$, namely $020$ and $120$, whose indices are $2$ and $6$. Since $R(8)=(1,0)$, $R(2)=(0,0)$, and $R(6)=(1,1)$, the two row-feature transitions are $(1,0)\to(0,0)$ and $(1,0)\to(1,1)$. Likewise, $C(8)=(1,0)$ and $C(2)=C(6)=(0,0)$, so both column-feature transitions are $(1,0)\to(0,0)$. Every other row of the table is obtained in exactly the same way.

\medskip
\noindent\textit{Step 4: collection and exhaustiveness of the transition types.}
Grouping the $24$ directed transitions in the preceding table gives the following multiplicities.
\begin{center}
\begingroup\scriptsize
\setlength{\tabcolsep}{2pt}
\renewcommand{\arraystretch}{1.12}
\begin{tabular}{c|rrrrrr|r}
\hline
feature transition & $(0,0)\to(1,0)$ & $(0,1)\to(1,0)$ & $(1,0)\to(0,0)$ & $(1,1)\to(0,0)$ & $(1,0)\to(1,1)$ & $(0,0)\to(0,1)$ & total \\
\hline
row feature    & 8 & 2 & 8 & 2 & 2 & 2 & 24 \\
column feature & 8 & 2 & 8 & 2 & 2 & 2 & 24 \\
\hline
\end{tabular}
\endgroup
\end{center}
Each of the six displayed types occurs. Moreover, the multiplicities in either feature system sum to $24$, which is the total number of transitions established in Step~2. Hence no additional row- or column-feature transition is possible, and in both cases the set of possible transitions is exactly
\[
\{(0,0)\to(1,0),\ (0,1)\to(1,0),\ (1,0)\to(0,0),\ (1,1)\to(0,0),\ (1,0)\to(1,1),\ (0,0)\to(0,1)\}.
\]
This completes the verification used in Lemma~\ref{lemma no monochromatic bond}.

\end{spacing}
\end{document}